\documentclass[10pt]{article}
\usepackage{amsmath}
\usepackage{amsfonts}
\usepackage{amsthm}

\newtheorem{thm}{Theorem}[section]
\newtheorem{lem}[thm]{Lemma}
\newtheorem{prop}[thm]{Proposition}
\newtheorem{cor}[thm]{Corollary}
\newtheorem{hyp}[thm]{Hypothesis}

\newtheorem*{thma}{Theorem A}
\newtheorem*{thmb}{Theorem B}
\newtheorem*{thmc}{Theorem C}
\newtheorem*{thmd}{Theorem D}
\newtheorem*{thme}{Theorem E}
\newtheorem*{thmf}{Theorem F}

\newcommand{\Hpi}{\mathrm{Hall}_\pi}
\newcommand{\Hpip}{\mathrm{Hall}_{\pi'}}
\newcommand{\ccl}{\mathrm{ccl}}
\newcommand{\pcl}{\mathrm{pr}}
\newcommand{\uip}{_{\lhd i+1}}
\newcommand{\ui}{_{\lhd i}}
\newcommand{\un}{_{\lhd}}
\newcommand{\uo}{_{\lhd 0}}
\newcommand{\ut}{_{\lhd 2}}
\newcommand{\upt}{^{\rhd}}
\newcommand{\angy}{\langle y \rangle}

\newcommand{\Irr}{\mathrm{Irr}}
\newcommand{\Conpi}{\mathrm{Con}_{\pi}}
\newcommand{\SConpi}{\mathrm{SCon}_{\pi}}
\newcommand{\krn}{\mathrm{ker}}

\newcommand{\Lev}{\mathrm{Lev}_G}
\newcommand{\Aut}{\mathrm{Aut}}
\newcommand{\Sym}{\mathrm{Sym}(4)}
\newcommand{\bP}{\mathbb{P}}
\newcommand{\cO}{\mathrm{O}}

\begin{document}

\title{A problem in the Kourovka notebook concerning the number of conjugacy classes of a finite group}

\author{Colin Reid\\
School of Mathematical Sciences\\
Queen Mary, University of London\\
Mile End Road, London E1 4NS\\
c.reid@qmul.ac.uk}

\maketitle

\begin{abstract}In this paper, we consider Problem 14.44 in the Kourovka notebook, which is a conjecture about the number of conjugacy classes of a finite group.  While elementary, this conjecture is still open and appears to elude any straightforward proof, even in the soluble case.  However, we do prove that a minimal soluble counterexample must have certain properties, in particular that it must have Fitting height at least $3$ and order at least $2000$.\end{abstract}


\section{Introduction}

\paragraph{Notation} We will denote by $\bP$ the set of prime numbers, $\pi$ some subset thereof, and $\pi'$ its complement in $\bP$.  Given groups $H \leq K$, and $k$ a subset of $K$, we will use $H_k$ to mean the centraliser in $H$ of $k$, and $k^H$ to mean $\{k^h | h \in H\}$.  Given $k \in K$, we write $H_k$ for $H_{\{k\}}$ and $k^H$ for $\{k\}^H$.  The number of conjugacy classes of the group $H$ is denoted $\ccl(H)$, and $\ccl_\pi(H)$ is the number of conjugacy classes of $\pi$-elements.  If $\Gamma$ is a set of subgroups of $K$, then $\Gamma \cap H$ means $\{ S \cap H | S \in \Gamma \}$.  The derived subgroup of $H$ will be denoted $H'$.  Given a finite soluble group $G$ and a set of primes $\pi$, $\cO_\pi(G)$ is the largest normal $\pi$-subgroup of $G$.\\

This paper concerns the following open conjecture, part (i) of which is Problem 14.44 in the Kourovka notebook (\cite{1}):

\paragraph{Conjecture}Let $G$ be a finite group with subgroups $A$ and $B$ such that $G = AB$ and $(|A|,|B|)=1$.  Then:

(i) $\ccl(G) \leq \ccl(A) \ccl(B)$;

(ii) $\ccl(G) = \ccl(A) \ccl(B)$ if and only if $G \cong A \times B$.\\

As noted in \cite{1}, if the condition $(|A|,|B|)=1$ is relaxed to $A \cap B = 1$, the conjecture becomes false, with a counterexample given by $G$ dihedral of order $4pq$ where $p$ and $q$ are distinct primes,$A$ dihedral of order $2p$, and $B$ dihedral of order $2q$.  On the other hand, it is clear that $\ccl(G) \leq \ccl(H)|G:H|$ for any subgroup $H$, so certainly $\ccl(G) \leq \ccl(A)\ccl(B)$ if either $A$ or $B$ is abelian and $A \cap B = 1$.\\

This conjecture has received very little attention in the existing literature.  The most important published contributions to date can be found in a paper of Gallagher (\cite{2}), and a subsequent paper by Vera L\'{o}pez and Ortiz de Elguea (\cite{3}).  Although these papers precede the appearance of Problem 14.44 in the Kourovka Notebook and do not consider it directly, they are concerned with enumeration of conjugacy classes in a manner relevant to this problem and effectively prove the conjecture in certain special cases.  The main result arising from these papers is as follows:

\begin{thm}\label{nsgpconpi}(\cite{2}, in the case $\pi = \bP$; \cite{3} for other $\pi$) Let $G$ be a finite group, let $N \unlhd G$, and let $\pi$ be a set of primes.  Then $\ccl_{\pi}(G) \leq \ccl_{\pi}(N)\ccl_{\pi}(G/N)$.  Equality occurs if and only if $h \in NG_g$ whenever $g$ is a $\pi$-element and $h$ is any element of $G$ for which $[g,h] \in N$.\end{thm}

Gallagher's paper also gives a character-theoretic proof in the case $\pi = \bP$, with a different condition for equality: equality occurs if and only if each irreducible character of $N$ extends to $N \langle g,h \rangle$ for every $g,h \in G$ such that $[g,h] \in N$.\\

A natural context in which to consider the conjecture is that of finite soluble groups.  For a detailed account of what is known about finite soluble groups, see \cite{4}.

\paragraph{Definitions} Let $G$ be a finite soluble group.  A \emph{Hall $\pi$-subgroup} is a subgroup whose order is a $\pi$-number and whose index is a $\pi'$-number.  The set of Hall $\pi$-subgroups of $G$ is denoted $\Hpi(G)$.  Note that the conditions $G=AB$ and $(|A|,|B|)=1$ taken together are equivalent to requiring $A$ to be a Hall $\pi$-subgroup and $B$ a Hall $\pi'$-subgroup of $G$ for some $\pi$.  A \emph{Sylow system} for $G$ is a lattice $\Sigma$ of pairwise permutable Hall subgroups of $G$ such that $\Sigma$ contains exactly one element of $\Hpi(G)$ for any $\pi$.  A \emph{complement basis} for $G$ is a set of Hall $\{p\}'$-subgroups of $G$, one for each prime $p$ dividing the order of $G$.\\

\begin{thm}[Hall's theorem]
Let $G$ be a finite soluble group.  The set $\Hpi(G)$ is nonempty for every $\pi \subseteq \bP$, and every $\pi$-subgroup of $G$ is contained in a Hall $\pi$-subgroup.  Every complement basis of $G$ is contained in a unique Sylow system $\Sigma$ of $G$ (in fact, the complement basis generates $\Sigma$ as a lattice).  Any lattice of pairwise permutable Hall subgroups is contained in at least one Sylow system.  $G$ acts transitively on the set of Sylow systems of $G$ (or equivalently on the set of complement bases) via
\[ \Sigma^g := \{S^g|S \in \Sigma\}.\]
\end{thm}

As a consequence, any finite soluble group $G$ will admit a suitable factorisation $G=AB$, for any $\pi \subseteq \bP$, and the isomorphism types of $A$ and $B$ are determined uniquely by the isomorphism type of $G$.  We will make frequent use of the following hypothesis, which will generally be assumed whenever subgroups $A$ and $B$ of $G$ are referred to:\\

\begin{hyp}\label{solhyp}$\pi$ is some set of primes, $G$ is a finite soluble group, $\Sigma$ is a Sylow system of $G$, and $A$ and $B$ are elements of $\Sigma$ such that $A \in \Hpi(G)$ and $B \in \Hpip(G)$.
\end{hyp}

For convenience, we also define two properties of finite soluble groups relating to the conjecture.

\paragraph{Definition}We say $G$ is $\Conpi$ if $\ccl(G) \leq \ccl(A)\ccl(B)$, in other words $G$ agrees with part (i) of the conjecture, and $G$ is $\Conpi^*$ if either $\ccl(G) < \ccl(A)\ccl(B)$ or $G \cong A \times B$, that is, $G$ agrees with both parts of the conjecture.\\

It is immediate that a group is $\Conpi$ if and only if it is $\mathrm{Con}_{\pi'}$, and similarly for $\Conpi^*$.  Nilpotent groups are clearly $\Conpi^*$ for every $\pi$, and by Gallagher's result, a group with a normal Hall $\pi$-subgroup is $\Conpi$.  Our aim is to generalise these results in the soluble case by imposing weaker conditions on the group $G$.\\

Both parts of the conjecture remain open for finite soluble groups (indeed, no counterexamples of any sort are known).  Although the conjecture is elementary, and concerns a basic invariant in finite group theory, it does not appear to admit an easy proof.  One approach which can be seen to fail is as follows:\\

Given an element $g \in G$, it can be written in a unique way as $xy$, where $x \in A$ and $y \in B$.  We might compare $|G_{xy}|$ with $|A_x||B_y|$, given the following:
\[|G|\ccl(G) = \sum_{g \in G}|G_g| = \sum_{x \in A}\sum_{y \in B}|G_{xy}|\]
\[|G|\ccl(A)\ccl(B) = |A||B|\ccl(A)\ccl(B) = \sum_{x \in A}\sum_{y \in B}|A_x||B_y|\]
However, if $x$ and $y$ do not commute, the centralisers of $x$ and $y$ may bear little relation to the centraliser of $xy$.  Even if $xy=yx$, this only ensures $A_{xy} \leq A_x$ and $B_{xy} \leq B_y$; it does not mean that $A_{xy}$ and $B_{xy}$ are Hall subgroups of $G_{xy}$.  So it is possible that $|G_{xy}|>|A_{xy}||B_{xy}|$, or indeed $|G_{xy}|>|A_x||B_y|$.  An example of this is as follows:\\
 
Let $G$ be a group of the form $(H \times K) : \Sym$, where $H$ is elementary abelian of order $p^4$ and $K$ is elementary abelian of order $q^4$, for $p$ and $q$ distinct primes such that $p>3$ and $q>2$, such that $\Sym$ permutes generators $\{e_1,e_2,e_3,e_4\}$ of $H$ and $\{f_1,f_2,f_3,f_4\}$ of $K$.  Let $\pi=\{2,p\}$, let $x=e_1e^{-1}_2$, and let $y=f_1f^{-1}_2$.  Then $x$ and $y$ commute, and indeed have the same centraliser in $G$, a subgroup of index $12$.  So $|G_{xy}|=2p^4 q^4$.  But if we take $A$ to be $p^4:D$, where $D$ is the subgroup of $Sym(4)$ generated by $(12)(34)$ and $(1234)$, we see $D_x$ is trivial (since $D$ does not contain the element $(34)$) and so $A_x$ has order $p^4$.  Any choice for $B$ will give $B_y$ of order $q^4$, so $|A_x||B_y|=p^4 q^4 < |G_{xy}|$.\\

The example above also illustrates that $A_x$ need not be a Hall subgroup of $G_x$ even for $x \in A$, something which also occurs within $\Sym$ itself: $(12)(34)$ has a centraliser of order $8$ in $\Sym$, but is not centralised by $D$.\\
 
One might also try to proceed by considering a finite soluble group $G$, all of whose proper sections are $\Conpi^*$ (where `section' is used here to indicate a quotient of a subgroup, with no further restrictions).\\

An elementary approach along these lines would be to take a normal subgroup $N$ of prime index (let us suppose $B \leq N$), observe that $\ccl(N) \leq \ccl(N \cap A)\ccl(B)$, and hope that
\[ f(\ccl(G))-f(\ccl(N)) \leq f(\ccl(A)) - f(\ccl(N \cap A)) \]
where $f$ is some increasing function.  This argument can be seen to fail in the case where $G$ and $N$ have shapes $7:6$ and $7:2$ respectively, and $\pi=\{3,7\}$: $\ccl(N) = 5$ and $\ccl(G) = 7$, so $\ccl(G) > \ccl(N)$, but $\ccl(N \cap A)=7$ and $\ccl(N) = 5$, so $\ccl(A) < \ccl(N \cap A)$.\\

This example illustrates another complication: given a group $G$ and a subgroup $H$, it is possible for $\ccl(H)$ to be greater than $\ccl(G)$.  The importance of this obstacle is illustrated by the following, which is easily proved:

\begin{prop}\label{easysgp}
Assume Hypothesis \ref{solhyp}.  Suppose that $\ccl(S_x) \leq \ccl(A)$ for every $\pi$-element $x$, where $S_x$ is a Hall $\pi'$-subgroup of $G_x$ (and hence isomorphic to a subgroup of $B$, by Hall's theorem).  Then $G$ is $\Conpi$.  If in addition $\ccl(S_x) < \ccl(B)$ whenever $S_x$ is not isomorphic to $B$, then $G$ is $\Conpi^*$.\end{prop}

We are also able to prove the following special case, as an elementary consequence of Hall's theorem:

\begin{thma} Let $G$ be a finite soluble group, and suppose $G/(\cO_{\pi}(G)\cO_{\pi'}(G))$ has a normal Hall $\pi$-subgroup.  Then $G$ is $\Conpi^*$.\end{thma}

However, the erratic behaviour of the orders of centralisers when passing to subgroups, even normal subgroups and Hall subgroups, means there is limited scope for further results by purely elementary means.  More promising is to consider $\ccl(G)$ as the number of irreducible complex characters of the group.

\paragraph{Notation}Given a finite group $H$, we denote by $\Irr(H)$ the set of irreducible complex characters of $H$.  Given a subgroup $K$ of $H$ and $\theta \in \Irr(K)$, the set $\Irr(H|\theta)$ consists of all $\chi \in \Irr(H)$ such that $[\chi_K, \theta] \neq 0$, where $\chi_K$ is the restriction of $\chi$ to $K$.  We define an action of $\mathrm{N}_H (K)$ on $\Irr(K)$, given by $\theta^h(k^h) = \theta(k)$ for all $k \in K$ and $h \in \mathrm{N}_H (K)$.  The \emph{centraliser} $H_{\theta}$ of $\theta$ is the set
\[\{h \in H|K^h = K, \theta^h = \theta\} \]
and $\Irr_H(K)$ consists of all $\theta \in \Irr(K)$ such that $H_{\theta} = \mathrm{N}_H (K)$.  Similarly, if $T$ is a set of automorphisms of $H$, \[\Irr_T(H) = \{ \chi \in \Irr(H) | \chi^t = \chi \; \forall t \in T \}\]

Characters are at least somewhat well-behaved when passing to normal subgroups, thanks to Clifford's theorem.  There are also character-theoretic results concerning coprime action, the most important of which for this paper is a correspondence due to Glauberman (\cite{5}): for every pair of groups $G$ and $S$ such that $S$ is soluble and acts on $G$ and $(|G|,|S|)=1$, there is a canonical bijection from $\Irr_S(G)$ to $\Irr(G_S)$.  (See \cite{5} or Chapter 13 of \cite{6} for a detailed description of this bijection.  A more general form of this correspondence, incorporating work of Isaacs, is now also known as the Glauberman-Isaacs correspondence.)\\
  
Using this correspondence, we obtain the following:

\begin{thmb}
Let $G$ be a finite soluble group and let $N \lhd G$ such that $G/N$ is cyclic and $(|N|,|G/N|)=1$.  Suppose that every subgroup of $N$ (including $N$ itself) is $\Conpi$.  Then $G$ is $\Conpi$.  If in addition $N$ is $\Conpi^*$, then $G$ is $\Conpi^*$.
\end{thmb}

Another situation in which character theory will pay dividends is when $N$ is an elementary abelian normal subgroup of $G$ satisfying certain extra conditions, the most important of which is given below.

\paragraph{Definition}Let $N$ be an elementary abelian normal subgroup of $G$; we may suppose without loss that $N \leq A$.  Say $G$ has \emph{balanced action} on $\Irr(N)$ if $A_\nu$ is a Hall $\pi$-subgroup of $G_\nu$ for all $\nu \in \Irr(N)$.  (We make a similar definition if $N \leq B$.  Note that thanks to Hall's theorem, the choice of $A$ and $B$ is unimportant.)\\

Note that in every $G$-orbit of $\Irr(N)$, there exists some $\nu$ such that $A_\nu$ is a Hall $\pi$-subgroup of $G_\nu$, and if $N$ is of prime order, all non-trivial characters have the same centraliser, namely the centraliser of $N$ in $G$.  So in particular, if $N$ is cyclic then $G$ automatically has balanced action on $\Irr(N)$.\\

An example of imbalanced action is given by $G = \mathrm{Sym}(4)$, with $\pi=\{2\}$, and $N$ the normal subgroup of order $4$.  The three non-trivial $\nu \in \Irr(N)$ are conjugate in $G$, each with centraliser containing a Hall $\pi$-subgroup of $G$, but only one of these is invariant in $A$.\\

For the next theorem, we will also need to define a new property of finite soluble groups.  Say $G$ is $\SConpi$ if it satisfies the following conditions:

(i) $G$ is $\Conpi$;
 
(ii)whenever $N$ is a central subgroup of $G$ of order $p \in \pi$, the following holds:

\[ \ccl(G) - \ccl(G/N) \leq (\ccl(A) - \ccl(A/N))\ccl(B) \]

(iii)whenever $N$ is a normal subgroup of $G$ of order $p \in \pi'$, an analogous condition to (ii) holds.\\

As with $\Conpi$, the following question is currently unanswered:

\paragraph{Question}Are there any finite soluble groups which are not $\SConpi$?

\begin{thmc}
Let $G$ be a finite soluble group, and $N$ an elementary abelian normal subgroup of $G$.  Suppose $G$ has balanced action on $\Irr(N)$, and that $G_{\nu}/N$ is $\Conpi$ for every $\nu \in \Irr(N)$.  Suppose also that one of the following holds:\\
 
(i)Every $\nu \in \Irr(N)$ extends to $G$.  (For this to occur, it is sufficient that $N$ be complemented in $G$.)\\

(ii)For every $\nu \in \Irr(N)$, the group $H=G_\nu/\krn(\nu)$ is $\SConpi$.\\

Then $G$ is $\Conpi$.  If $G/N$ is $\Conpi^*$, then $G$ is $\Conpi^*$.
\end{thmc}

The conditions for $\SConpi$ may be stronger than $\Conpi$ in general.  However, if $Z(G) \cap G'=1$, we can obtain $\SConpi$ from $\Conpi$.  Note that if $G$ has trivial centre, $G$ is evidently $\SConpi$ if and only if it is $\Conpi$.

\begin{thmd}
Let $G$ be a finite soluble group.  Suppose that for some central subgroup $M$ of prime order, every subgroup of $G/M$ is $\Conpi$ and $M \cap G' = 1$.  Then $G$ is $\Conpi$.  If in addition $Z(G) \cap G'=1$, then $G$ is $\SConpi$.  If $G/M$ is $\Conpi^*$ then $G$ is $\Conpi^*$.
\end{thmd}

We also give a possible inductive approach to showing $\mathcal{K}$-groups are $\Conpi$, where $\mathcal{K}$ is a class of finite soluble groups closed under taking quotients:

\begin{thme}Let $\mathcal{K}$ be a class of finite soluble groups closed under quotients.  Suppose that any $\mathcal{K}$-group $G$ satisfies the following:\\
 
Assume Hypothesis \ref{solhyp}.  The number of faithful irreducible complex characters of $G$ is at most $|\Irr_G(A \times B)|$, where $\Irr_G(A \times B)$ consists of those irreducible complex characters $\chi$ of $A \times B$ such that $\krn(\chi)$ does not contain any subgroups of $A$ or $B$ that are normal in $G$.\\

Then every $\mathcal{K}$-group is $\SConpi$.
\end{thme}

The hypotheses of Theorem E give rise to the following question:

\paragraph{Question}Under Hypothesis \ref{solhyp}, is it possible for the number of faithful irreducible complex characters of $G$ to be greater than $|\Irr_G(A \times B)|$?

A negative answer to this question would imply that every finite soluble group is $\SConpi$ for every $\pi$.\\

Here is a summary of what has been shown about minimal soluble counterexamples to part (i) of the conjecture.

\begin{cor}
Let $G$ be a finite soluble group, such that $G$ is not $\Conpi$, but every proper section of $G$ is $\Conpi$.  Then:
 
(i) $G/(\cO_{\pi}(G)\cO_{\pi'}(G))$ does not have a normal Hall $\pi$-subgroup or a normal Hall $\pi'$-subgroup.
 
(ii) Every prime for which $G$ has a non-trivial cyclic Sylow $p$-subgroup divides the order of $G'$.
 
(iii) Let $N$ be an elementary abelian normal subgroup of $G$.  If $G$ has balanced action on $\Irr(N)$, then there exist $\nu,\nu' \in \Irr(N)$ such that $\nu$ does not extend to $G_{\nu}$ and $H=G_{\nu'}/\krn(\nu')$ fails to be $\SConpi$.
 
(iv)The centre of $G$ is contained in $G'$.

(v) $G$ has an image $K$ which is an irreducible linear group, such that $K$ fails to satisfy the conditions of Theorem E.
\end{cor}

As a demonstration of the applicability of the results in this paper, and as empirical evidence for the conjecture, we prove the following:

\begin{thmf}
Let $\pi$ be a set of primes, and let $G$ be a finite soluble group which is a minimal counterexample to the assertion `every finite soluble group is $\Conpi^*$'.  Then $|G|>2000$.
\end{thmf}

This could proved by direct calculation, as the isomorphism groups of order up to $2000$ (with the exception of groups of order $1024$, which are obviously $\Conpi^*$) have been explicitly enumerated by Besche, Eick and O'Brien (\cite{7}).  Such a database is available in the \emph{MAGMA Computational Algebra System} and its Small Groups Library (\cite{8}), and therefore Theorem F could be verified using this system alone.  But in fact, all groups of order up to $2000$ can be ruled out by hand using the previous theorems, with the exception of the following group orders:
\[ 336, 672, 1008, 1200, 1296, 1344, 1680 \]
There are fewer than $20000$ groups of these orders, which can therefore be checked directly and relatively quickly by use of \cite{8}.  No counterexamples to $\Conpi^*$ were found.\\

Finally, a remark about $\pi$-special characters.  Although none of the new results in this paper are obtained using $\pi$-special characters, they appear to be closely related to the problem at hand, and so deserve mention.  These characters are discussed in detail in Chapter VI of \cite{9}, with all the necessary references.

\paragraph{Definitions}Let $G$ be a finite group, with $\chi \in \Irr(G)$, and $\pi$ a set of primes.  Then $\chi$ is \emph{$\pi$-special} if:

(i) $\chi(1)$ is a $\pi$-number;

(ii) Given any subnormal subgroup $S$ of $G$, and $\theta \in \Irr(S)$ such that $[\chi_s,\theta] \not= 0$, the determinantal order of $\theta$ is a $\pi$-number.\\

We say $\chi \in \Irr(G)$ is \emph{$\pi$-factorable} if it can be expressed as the product of a $\pi$-special character and a $\pi'$-special character. 

\begin{thm}Let $G$ be a finite $\pi$-separable group, and let $\chi \in \Irr(G)$.  Let $A$ be a Hall $\pi$-subgroup of $G$.

(i) If $\chi$ is $\pi$-factorable, its factorisation is unique.
 
(ii) If $\chi$ is primitive, it is $\pi$-factorable.
 
(iii) Restriction of characters defines an injective map from the $\pi$-special characters of $G$ to $\Irr(A)$.
\end{thm}

An obvious consequence of this is that we have a canonical way of sending the $\pi$-factorable characters of $G$ to $\Irr(A \times B)$, where $B$ is a Hall $\pi'$-subgroup of $G$, such that the map so defined is injective.  In particular:

\begin{cor}Let $G$ be a $\pi$-separable group, with Hall $\pi$-subgroup $A$ and Hall $\pi'$-subgroup $B$.  Then the number of $\pi$-factorable characters of $G$ is at most $\ccl(A)\ccl(B)$.\end{cor}

If a method could be found to estimate the number of non-$\pi$-factorable characters of $G$ (which are necessarily imprimitive) by similar means, this could answer the conjecture for $\pi$-separable groups, or at least provide an inequality resembling that of the conjecture.  However, no such method is known, even under the assumption that $G$ is soluble.

\section{Centralisers}

In this section, we obtain Theorem A more or less directly from Hall's theorem together with some elementary centraliser formulae for $\ccl(G)$.  The most basic of these is the following:
\begin{equation}
\label{basic}
\ccl(G) = \frac{1}{|G|} \sum_{g \in G} |G_g|
\end{equation}

Given subsets $H$ and $K$ of $G$, we use $(H,K)$ to denote the set of ordered pairs of elements $(h,k)$ such that $h \in H$ and $k \in K$, and $hk = kh$.  If $H$ and $K$ are subgroups of $G$, we define a function to indicate how often elements of $H$ and $K$ commute, taking values between 0 and 1:
\[ \pcl(H,K) := \frac{1}{|H||K|}\sum_{h \in H} |K_h| \]

Note that $\pcl(H,K) = \pcl(K,H) = |(H,K)|/|H||K|$ and $\pcl(G,G) = \ccl(G)/|G|$.

\begin{lem} Let $G$ be a finite group, let $H$ and $L$ be subgroups of $G$, and let $K < L$.  Then
\[ \frac{\pcl(H,K)}{|L:K|} < \pcl(H,L) \leq \pcl(H,K) \]
For the second inequality, equality occurs if and only if $L = \bigcap_{h \in H}(L_h K)$.
\end{lem}

\begin{proof}Clearly $(H,K) \subseteq (H,L)$, and in fact the containment is strict, as $(H,L)$ also contains $(1,l)$ for any $l \in L \setminus K$.  This gives the first inequality.

For any element of $H$, we have $|h^K| \leq |h^L|$, with equality if and only if $L = L_h K$.  By the Orbit-Stabiliser Theorem, this means $|L_h : K_h| \leq |L:K|$ with the same condition for equality.  This gives the second inequality and the stated condition for equality.
\end{proof}

\begin{cor}\label{sgpccl} Let $G$ be a finite group with $K < G$.  Then
\[ \frac{\ccl(K)}{|G:K|} < \ccl(G) \leq \ccl(K)|G:K| \]
\end{cor}

Every element of a finite group $G$ can be written in a unique way as $xy$ where $x$ is a $\pi$-element, $y$ is a $\pi'$-element and $xy=yx$.  Thus we can identify $G$ with the set of pairs $(x,y)$ of elements of $G$ satisfying these conditions, and consider the conjugation action of $G$ on such pairs.  We obtain the following formula by considering for each $\pi$-element $x$ the orbits occurring in $(x,Q)$, where $Q$ is the set of $\pi'$-elements commuting with $x$:
\begin{equation}
\label{pp1}
 \ccl(G) = \sum_{r \in R} \ccl_{\pi'}(G_r)
\end{equation}
where $R$ is a set of representatives of the conjugacy classes of $\pi$-elements of $G$.\\

Now specialise to the case of finite soluble groups.
  
\paragraph{Definition}Let $G$ be a finite soluble group, and let $\Sigma$ be a Sylow system for $G$.  Let $H$ be a subgroup of $G$.  We say $\Sigma$ \emph{reduces into} $H$ (written $\Sigma \searrow H$) if $S \cap H$ is a Hall $\pi$-subgroup of $H$ for every $\pi \subseteq \bP$, where $S$ is the Hall $\pi$-subgroup of $G$ occurring in $\Sigma$.\\

The following is easily obtained from Hall's theorem and the definition of a Sylow system:

\begin{lem}\label{xilem}Let $\Sigma$ be a Sylow system for the finite soluble group $G$.

(i) Let $g \in G$ such that $\Sigma \searrow G_g$.  Then $g$ is contained in every $S \in \Sigma$ for which $|g| \mid |S|$.

(ii) Let $H$ be any subgroup of $G$.  Then there exists $g \in G$ such that $\Sigma \searrow H^g$.

(iii) Suppose $\Sigma \searrow H$.  Then $\Sigma \cap H$ is a Sylow system of $H$.  For any subgroup $L$ of $H$, this Sylow system of $H$ reduces into $L$ if and only if $\Sigma$ reduces into $L$.\end{lem}

For the remainder of this paper, Hypothesis \ref{solhyp} is assumed unless otherwise stated.\\

We can now demonstrate Proposition \ref{easysgp}.  By Hall's theorem, $\ccl_{\pi'}(G_r)$ is at most $\ccl(K)$, where $K$ is a Hall $\pi'$-subgroup of $G_r$.  The weaker condition in Proposition \ref{easysgp} now makes the assumption that $\ccl(K) \leq \ccl(B)$.  So we have:
\[\ccl(G) \leq \sum_{r \in R} \ccl(B) \]
and again by Hall's theorem, $|R| \leq \ccl(A)$.\\

Now suppose that for every $\pi'$-element $y$ of $G$ and $S_y$ a Hall $\pi$-subgroup of $G_y$, we have either $S_y$ isomorphic to $B$ or $\ccl(S_y) < \ccl(B)$.  Then the only way it is possible for $\ccl(G)$ to equal $\ccl(A)\ccl(B)$ is if $G_r$ contains a Hall $\pi'$-subgroup of $G$ for every $r \in R$.  By choosing the representatives $r$ such that $\Sigma \searrow G_r$, we can ensure $R \subseteq A$ and $B$ centralises $R$.  Every conjugate of an element of $A$ is of the form $r^{ba}=r^a$ for some $a \in A$, $b \in B$ and $r \in R$.  Hence $A$ is normal in $G$, which means $G$ is $\Conpi^*$ by Theorem \ref{nsgpconpi} (or by Theorem A, the proof of which does not use Proposition \ref{easysgp}).\\

We will also use a slightly different form of Eq. \eqref{pp1}.  Given finite groups $H \leq K$ and $h \in H$, define $f^K_H (h)$ to be the number of distinct conjugacy classes of $H$ found in the set $h^K \cap H$.  Note that if $k \in K$ and $h^k \in H$, we always have $f^K_H (h^k) = f^K_H (h)$, a fact which will be useful later.  Since $G$ is soluble, in \eqref{pp1} we can insist $R \subset A$, giving the following:
\begin{equation}
\label{pp2}
 \ccl(G) = \sum_{t \in T} \frac{\ccl_{\pi'}(G_t)}{f^G_A (t)}
\end{equation}
where $T$ is a set of representatives of the conjugacy classes of $A$.\\

Interestingly, if $G$ has a normal Hall $\pi$-subgroup, it is guaranteed to satsify the stronger condition in Proposition \ref{easysgp}, a fact which has consequences for a wider class of groups.  This can be derived from the Glauberman correspondence.  However, we do not need the full strength of the correspondence at this stage, so instead a proof is given below that is elementary with the assumption of Hall's Theorem.

\begin{lem}\label{alhdg}Suppose $A \lhd G$, and let $y \in B$.  Then $\ccl(A_y) \leq \ccl(A)$, with equality if and only if $A_y = A$.
\end{lem}

\begin{proof}Without loss of generality, we may assume $B=\angy$.  Note that under this assumption, all elements of $B$ lie in distinct conjugacy classes of $G$.\\

Using Eq. \eqref{pp1}, we obtain
\[ \ccl(G) \leq \ccl_{\pi}(G)|y| \]
with equality if and only if $B_r = B$ for every representative $r$ of the $\pi$-conjugacy classes of $G$.  As this holds regardless of choice of representatives, equality occurs if and only if $A_y = A$.\\

If $|y|$ is prime, applying \eqref{pp1} again with $\{|y|\}$ in place of $\pi$ gives
\[ \ccl(G) = \ccl_{\pi}(G) + (|y|-1)\ccl(A_y) \]
Combining this with the previous inequality gives $\ccl(A_y) \leq \ccl_\pi(G)$, with equality if and only if $A_y = A$.  Finally, it is clear that $\ccl_\pi(G) \leq \ccl(A)$ and equality occurs if $A_y = A$.\\

If $|y|$ is not prime, we proceed by induction on $|G|$.  Choose an integer $n$ so that $|y^n|$ is prime.  Note that $A_{y^n}$ is precisely the set of elements of $A$ whose orbits under $y$ have size dividing $n$, so $A_{y^n}$ is normalised by $y$.\\

If $A_{y^n} < A$, then by considering the proper subgroup $\langle A_{y^n},y \rangle$ of $G$, we obtain by induction
\[ \ccl(A_y) = \ccl((A_{y^n})_y)  \leq \ccl(A_{y^n}) \]
with equality if and only if $A_y = A_{y^n}$.

Since $|y^n|$ is prime, $\ccl(A_{y^n}) \leq \ccl(A)$, with equality if and only if $A_{y^n} = A$.

If $A_{y^n} = A$, then $Y = \langle y^n \rangle$ is central in $|G|$.  Considering the proper quotient $G/Y$ gives
\[ \ccl(A_y) = \ccl((AY/Y)_{yY})  \leq \ccl(AY/Y) = \ccl(A) \]
with equality if and only if $(AY/Y)_{yY} = AY/Y$, which holds if and only if $A_y = A$, because $Y \cap [A,B] = 1$.

Combining these inequalities gives the required result in all cases.\\
\end{proof}

We now define subgroups of $G$ analogous to the Fitting series of a finite soluble group, as follows:
\[ G\uo := 1 \; ; \quad \frac{G\uip}{G\ui} := \frac{\cO_\pi(G)\cO_{\pi'}(G)}{G\ui} \]
Where $i=1$, it will generally be omitted.\\

Write $A\ui$ for $A \cap G\ui$ and $B\ui$ for $B \cap G\ui$.  Note that $A\un=\cO_\pi(G)$ and $B\un=\cO_{\pi'}(G)$.  Define the \emph{height} of the product as follows:

Consider the least $i$ such that $A\ui = A$ or $B\ui = B$ (this will always occur; indeed $G\ui = G$ whenever $i$ is at least the Fitting height of $G$).  If $A\ui = A$ and $B\ui = B$, we say the product has height $i$; if exactly one of these equalities holds, we say the product has height $(i + \frac{1}{2})$.  So for example the direct product is the only product with height 1, and all (non-degenerate) semidirect products of $A$ and $B$ have height $1 \frac{1}{2}$.\\

Since $G$ is soluble, it is clear that the height of $G$ depends only on $\pi$ and the structure of $G$, not on the choice of $A$ and $B$.  So we may refer to the $\pi$-height $\mathrm{ht}_\pi(G)$ of $G$ without ambiguity.  Note that $\mathrm{ht}_\pi(G)$ is always at most the Fitting height of $G$, which is the motivation for the fractional numbering.  Our next aim is to prove the Conjecture for $G$ soluble satisfying $\mathrm{ht}_\pi(G) \leq 2 \frac{1}{2}$; this is Theorem A.\\

Define 
\[G\upt := \mathrm{N}_G(A) \cap \mathrm{N}_G(B) \; ; \quad A\upt := G\upt \cap A \; ; \quad B\upt := G\upt \cap B \]
From now on we will assume $\Sigma \searrow G\upt$: there will certainly be some Sylow system containing $A$ and $B$ for which this is the case, so we may assume that $\Sigma$ has been chosen accordingly.  Note that in this case $G\upt = A\upt \times B\upt $.\\

\begin{lem}Suppose $A\ut = A$, and let $x,y \in A$.  Then:

(i)$B = B\upt B\un$;

(ii) $B_x = {B\upt}_x {B\un}_x$;

(iii) if $x$ and $y$ are conjugate in $G$, they are also conjugate in $AB\upt$;

(iv) If $\Sigma \searrow G_x$, then $f^G_A (x) = |B\upt:{B\upt}_x|$.
\end{lem}

\begin{proof}
$A\ut = A$ means $AG\un/G\un \lhd G/G\un$, so $AB\un \lhd G$, so all conjugates of $A$ in $G$ lie in $AB\un$.  $B$ acts on the set of conjugates of $A$ by conjugation, and so does $B\un$; by Hall's theorem applied to $G$ and to $AB\un$, both actions are transitive.  Since $B\un \leq B$, it follows that $B = B\upt B\un$, as $B\upt$ is the stabiliser of $A$ in the action of $B$.\\

Given $b\upt \in B\upt$ and $b\un \in B\un$, we have $x^{b\upt} \in A$ and if $x^{b\upt b\un} \in A$, then $x^{b\upt}b\un = b\un x^{b\upt}$.  So $x^{b\upt b\un} = y \Rightarrow x^{b\upt} = y$.  This proves (ii) and (iii).\\

Suppose $\Sigma \searrow G_x$.  Then it follows from Lemma \ref{xilem} that $(AB\upt)_x = A_x {B\upt}_x$.  Since $B\upt$ acts as automorphisms of $A$, it can only fuse conjugacy classes of $A$ of equal size.  Hence 
\[f^{A{B\upt}}_A (x) = \frac{|x^{A{B\upt}}|}{|x^A|} = |x^{B\upt}| = |B\upt:{B\upt}_x|\]
$f^G_A (x) = f^{A{B\upt}}_A (x)$ by (iii).
\end{proof}

\begin{thma} Suppose $\mathrm{ht}_\pi(G) \leq 2 \frac{1}{2}$.  Then $G$ is $\Conpi^*$.
\end{thma}

\begin{proof} We assume $A\ut = A$.  Start with Eq. \eqref{pp2}:
\[ \ccl(G) = \sum_{t \in T} \frac{\ccl_{\pi'}(G_t)}{f^G_A (t)} \]
The summands are not affected by replacing each $t$ by some other element $z_t$, so long as $z_t$ is conjugate to $t$ in $G$ and $z_t \in A$.  By Lemma \ref{xilem}, we can do this in such a way that $\Sigma \searrow G_{z_t}$.\\

Given $y \in A$, define $B_{*y} = \bigcap_{z \in \angy} B^z$.  It is clear that $B_y B\un \leq B_{*y}$, so $|B:B_{*y}| \leq |B\upt:{B\upt}_y|$.  Since $y$ normalises $B_{*y}$, it follows from Lemma \ref{alhdg} that $\ccl(B_{*y}) \geq \ccl(B_y)$, with equality if and only if $B_{*y} = B_y$.\\

Now applying previous results:
\begin{align*}
\ccl(G) &= \sum_{t \in T} \frac{\ccl_{\pi'}(G_{z_t})}{f^G_A (z_t)} \\
&= \sum_{t \in T} \frac{\ccl_{\pi'}(G_{z_t})}{|B\upt:{B\upt}_{z_t}|} \\
&\leq \sum_{t \in T} \frac{\ccl(B_{z_t})}{|B\upt:{B\upt}_{z_t}|} \\
&\leq \sum_{t \in T} \frac{\ccl(B_{*z_t})}{|B\upt:{B\upt}_{z_t}|} \\
&\leq \sum_{t \in T} \frac{\ccl(B_{*z_t})}{|B:B_{*z_t}|} \\
&\leq \sum_{t \in T} \ccl(B) \\
&= \ccl(A) \ccl (B).
\end{align*}\\

For equality to occur, we must have equality at every stage.  In particular, for each $z_t$ we must have $\ccl(B_{z_t}) = \ccl(B_{*z_t})$, which means $B_{z_t} = B_{*z_t}$, and $\ccl(B_{*z_t})/|B:B_{*z_t}| = \ccl(B)$, which means that $B_{*z_t} = B$.  So
\[z^G_t = z^{BA}_t = z^A_t \subset A \]
for each $z_t$.  But the $z_t$ represent every conjugacy class of $\pi$-elements in $G$, so $B$ normalises $A$.  This means $B\ut = B$, so by the same argument again with the roles of $A$ and $B$ reversed,  $A$ normalises $B$.  Hence $G \cong A \times B$.
\end{proof}

\section{Character theory}

\subsection{Preliminaries}

Using character theory, in certain situations we can conclude that $G$ is $\Conpi$ from the fact that certain proper sections of $G$ are $\Conpi$, and similarly for $\Conpi^*$.\\

Here are some known results which will be used to prove Theorems B, C and D.  In the following, let $H$ be a finite group, with $N \lhd H$, and let $\nu \in \Irr(N)$ and $\eta \in \Irr(H|\nu)$.

\begin{thm}[Clifford's theorem]
 \[ \eta_N = e(\nu_1 + \dots + \nu_k) \]
 where $e$ is a positive integer and $\{\nu_1, \dots \nu_k\}$ is the set of conjugates of $\nu$ under the action of $H$.
\end{thm}

\begin{thm}
(i) There is a 1-1 correspondence between $\Irr(H|\nu)$ and $\Irr(H_\nu|\nu)$.
 
(ii) Suppose $\nu=\phi\theta$, for $\phi, \theta \in \Irr(N)$, such that $\phi$ and $\theta$ are invariant in $H_\nu$ and $\theta$ extends to $\chi \in \Irr(H_\nu)$.  Then $|\Irr(H_\nu|\nu)| = |\Irr(H_\nu|\phi)|$.

(iii)If $K$ is any subgroup of $G$ containing $N$, then $|\Irr(K|1_N)|=\ccl(K/N)$.
\end{thm}

\begin{proof}
(i) See Theorem 6.11 (b) of \cite{6}.
 
(ii) See Theorem 6.16 of \cite{6}.
 
(iii) $\Irr(K|1_N)$ consists of precisely those $\chi \in \Irr(K)$ for which $\krn(\chi) \geq N$.  These characters are in 1-1 correspondence with $\Irr(K/N)$ in an obvious way.
\end{proof}

\begin{cor}\label{irrcor}
\[ \ccl(H) = \sum_{\nu \in \Irr(N)}  \frac{|\Irr(H_\nu|\nu)|}{|H:H_\nu|}.\]
\end{cor}

\begin{thm}\label{extnthm}
Suppose one of the following holds:
 
(i) $H_\nu/N$ is cyclic;
 
(ii) $(|H_\nu/N|,|N|)=1$;

(iii) $N$ is abelian and complemented in $H_\nu$.
 
Then $\nu$ extends to $H_\nu$.
\end{thm}

\begin{proof}
(i) This is Corollary 11.22 of \cite{6}.
 
(ii) This is a special case of Corollary 6.28 of \cite{6}.
 
(iii) See Problem 6.18 of \cite{6}.
\end{proof}

\begin{prop}[Consequence of the Glauberman correspondence]\label{glaub}
Let $T$ be a soluble subgroup of $\Aut(H)$ such that $(|H|,|T|)=1$.  Then
\[ |\Irr_T(H)| = \ccl(H_T) \]
If $\Irr(H) = \Irr_T(H)$ then $T=1$.
\end{prop}

\subsection{Theorem B}

Given Hypothesis \ref{solhyp}, we can define a bijection $*$ from $G$ to $A \times B = G^*$ given by $(ab)^* = (a,b)$.  Similarly, given a subset $X$ of $G$, we may define $X^*$ to be the image of $X$ under $*$.  Consider now a subgroup $H$ of $G$.  In general, $H^*$ may not be a group.  In fact, $H^*$ is a group precisely if $H = (H \cap A)(H \cap B)$.  In this case we say $H$ is \emph{split} under the given factorisation of $G$.\\

For $H$ to split, it is certainly sufficient that $\Sigma \searrow H$.  Lemma \ref{xilem} thus ensures that every subgroup of $G$ is conjugate to a split subgroup.  In particular, all normal subgroups of $G$ are split, and moreover if $N \unlhd G$ then $N^* \unlhd G^*$.  Note also that if $N_1, N_2 \unlhd G$, then $(N_1 N_2)^* = N^*_1 N^*_2$.

\begin{lem}\label{extnlem}
Let $G$ and $*$ be as above, with $N \lhd G$ such that $G/N$ is cyclic and $(|G/N|,|N|)=1$.  Define
\[ \iota_k(N;G) = |\{\nu \in \Irr(N)||G_{\nu}:N| \geq k \}| \]
and suppose that
\[ \iota_k(N;G) \leq \iota_k(N^*;G^*) \quad \forall k \geq 0\]
Then $G$ is $\Conpi$, and the inequality given in the Conjecture is an equality if and only if
  \[ \iota_k(N;G) = \iota_k(N^*;G^*) \quad \forall k \geq 0\]
\end{lem}

\begin{proof}
Since every $\nu \in \Irr(N)$ extends to its centraliser $G_\nu$ and $G/N$ is abelian,
\[ \ccl(G) = \sum_{\nu \in \Irr(N)} \frac{\ccl(G_{\nu}/N)}{|G:G_{\nu}|}= \sum_{\nu \in \Irr(N)} \frac{|G_\nu:N|^2}{|G:N|}\]
by Corollary \ref{irrcor}, and a similar formula holds for $\ccl(A \times B)$.  The result now follows from the fact that $|G_\nu:N|^2/|G:N|$ is a strictly increasing function of $|G_\nu:N|$.
\end{proof}

\begin{thmb}
Let $G$ and $N$ be as in (ii) above with $G/N$ a $\pi$-group.  Suppose that every subgroup of $N$ (including $N$ itself) is $\Conpi$.  Then $G$ is $\Conpi$.  If in addition $N$ is $\Conpi^*$, then $G$ is $\Conpi^*$.
\end{thmb}

\begin{proof}
We shall prove that the inequalities given in the previous lemma hold in this case, and then consider conditions for equality.\\
 
Let $|G/N|=m$.  By induction on $|G|$, we may assume 
\[ \iota_k(N;G) \leq \iota_k(N^*;G^*) \quad \forall 0 \leq k < m \]
so we may restrict our attention to the inequality in the case $k=m$.  Since every $\nu \in \Irr(N)$ extends to $G_\nu$, but no further, and similarly for $G^*$, the following holds:
\[ \iota_m(N;G) = |\Irr_G(N)| \; ; \quad \iota_m(N^*;G^*) = |\Irr_{G^*}(N^*)| \]
By Hall's Theorem, $B \leq N$ and we can take $T \leq A$ such that $G$ is the semidirect product of $N$ by $T$ and $T$ normalises $B$.  By assumption $(|N|,|T|)=1$, and by induction on $|G|$ we may assume $T$ acts faithfully on $N$ and regard it as a subgroup of $\Aut(N)$.  Now applying Proposition \ref{glaub}:
\[ |\Irr_G(N)| = |\Irr_T(N)| = \ccl(L) \]
\[ |\Irr_{G^*}(N^*)| = |\Irr_T(N^*)| = \ccl(Y) \]
where $L$ is the centraliser in $N$ of $T$, and $Y = (L \cap A)\times B$.  Consider $t \in T$ acting on $x \in A$ and $y \in B$ by conjugation.  If $(xy)^t = xy$ then $x^t = x$ and $y^t = y$, since $T$ normalises both $A$ and $B$.  In other words, $L = (L \cap A)(L \cap B)$.  Since $L \leq N$, by assumption
\[ \ccl(L) \leq \ccl(L \cap A )\ccl(L \cap B) \]
Now $L \cap B = B_T$, so $\ccl(L \cap B) = |\Irr_T(B)| \leq |\Irr(B)|$ by Proposition \ref{glaub}.  Hence
\[\iota_m(N;G) = \ccl(L) \leq \ccl(L \cap A )\ccl(B) = \ccl(Y) = \iota_m(N^*;G^*) \]
This proves that $G$ is $\Conpi$.\\

Suppose now $\ccl(G) = \ccl(A)\ccl(B)$, and that $N$ is $\Conpi^*$.  Then by Lemma \ref{extnlem}
\[\ccl(N) = \iota_1(N;G) = \iota_1(N^*;G^*) = \ccl(N \cap A )\ccl(B) \]
which means $(N \cap A)$ centralises $B$.  We must also have $|\Irr_T(B)| = |\Irr(B)|$.  This means that $|\Irr_{\{t\}}(B)| = |\Irr(B)|$ for every $t \in T$, so by Proposition \ref{glaub}, $B_t = B$.  In other words, every element of $T$ centralises $B$.  Hence $A=(N \cap A)T$ centralises $B$, so $G \cong G^*$ as required.
\end{proof}

\subsection{Theorems C and D}

Corollary \ref{irrcor} gives a potential approach to the Conjecture for finite soluble groups in general.  Consider an abelian normal subgroup $N$, with $N \leq A$ say, and use the formula in Corollary \ref{irrcor} to compare $\ccl(G)$ and $\ccl(A \times B)$.  To prove $G$ is $\Conpi$ it suffices to show, for each $\nu \in \Irr(N)$:

\[ \frac{|\Irr(G_\nu|\nu)|}{|G:G_\nu|} \leq \frac{|\Irr(A_\nu|\nu)|\ccl(B)}{|A:A_\nu|}\]\\

There are however two difficulties here.  The first is that $A_\nu$ may not be a Hall $\pi$-subgroup of $G_\nu$.  The second is that there is in general no way of calculating $|\Irr(G_\nu|\nu)|$ exactly in a way that transfers easily to a calculation of $|\Irr(A_\nu|\nu)|$.\\

The first difficulty does not arise if we insist $G$ has balanced action on $N$, as defined in the Introduction.  In such a situation $|G:G_\nu| = |A:A_\nu||B|/|S_\nu|$, where $S$ is a Hall $\pi'$-subgroup of $G$ chosen to that $S_\nu$ is a Hall $\pi'$-subgroup of $G_\nu$, and the inequality reduces to the following:
\[|\Irr(G_\nu|\nu)||S_\nu| \leq |\Irr(A_\nu|\nu)|\ccl(B)|B|\]

Certainly $\ccl(B)|B|/|S_\nu| \geq \ccl(S_\nu)$, by Corollary \ref{sgpccl}.  So in fact, it suffices to show the following:
\begin{equation}
\label{baineq}
|\Irr(G_\nu|\nu)| \leq |\Irr(A_\nu|\nu)|\ccl(S_\nu)
\end{equation}

The second difficulty may be less serious, as a result in \cite{2} implies that $|\Irr(G_\nu|\nu)| \leq \ccl(G_\nu/N)$, and similarly for $|\Irr(A_\nu|\nu)|$, and the same paper also provides a method for calculating $|\Irr(G_\nu|\nu)|$ exactly.  However, the method described in $\cite{2}$ appears difficult to carry out in such a way that would allow easy comparison between $|\Irr(G_\nu|\nu)|$ and $|\Irr(A_\nu|\nu)|$ in a general situation; if such a comparison could be achieved, it could lead to a generalisation of Theorem C.\\
  
In any case, the difficulty is overcome if $\nu$ extends to $G_\nu$, as in this case $|\Irr(G_\nu|\nu)|=\ccl(G_\nu/N)$ and similarly for $A$.  It can also be overcome by considering the group $G_\nu/\krn(\nu)$, together with the assumption that this group is $\SConpi$, as defined in the introduction.\\

As a result we obtain the following:

\begin{thmc}
Let $G$ be a finite soluble group, and $N$ an elementary abelian normal subgroup of $G$.  Suppose $G$ has balanced action on $\Irr(N)$, and that $G_{\nu}/N$ is $\Conpi$ for every $\nu \in \Irr(N)$.  Suppose in addition that at least one of the following holds:\\
 
(i)Every $\nu \in \Irr(N)$ extends to $G$. (In particular, this will be the case if $N$ is complemented in $G$, by Theorem \ref{extnthm}.)\\

(ii)For every $\nu \in \Irr(N)$, the group $H=G_\nu/\krn(\nu)$ is $\SConpi$.\\

Then $G$ is $\Conpi$.  If $G/N$ is $\Conpi^*$, then $G$ is $\Conpi^*$.\\
\end{thmc}

\begin{proof}
Assume case (i), and assume $N \leq A$.  From the above discussion, we see that to show $G$ is $\Conpi$, it suffices to show, for each $\nu \in \Irr(N)$:
\[ \ccl(G_\nu/N) \leq \ccl(A_\nu/N)\ccl(S_\nu) \]

This inequality follows immediately from the hypothesis that $G_\nu/N$ is $\Conpi$.\\

Now assume case (ii).  Fix $\nu \not=1_N$; by balanced action, $G_\nu$ factorises as $A_\nu S_\nu$.  If $N$ is cyclic, the non-trivial irreducible characters of $\nu$ are transitively permuted by a Galois automorphism.  This implies $|\Irr(G_\nu|\nu)|$ and $|\Irr(A_\nu|\nu)|$ do not depend on the choice of $\nu \in \Irr(N) \setminus \{1_N\}$, so
\begin{equation}
\label{cyceq}
\ccl(G_\nu) = |\Irr(G_\nu|1_N)| + (p-1)|\Irr(G_\nu|\nu)| = \ccl(G_\nu/N) + (p-1)|\Irr(G_\nu|\nu)|
\end{equation}
and similarly for $A$.\\

More generally, given $\chi \in \Irr(G_\nu|\nu)$, then $\krn(\nu) \leq \krn(\chi)$, since by Clifford's theorem, $\chi_N$ is a multiple of $\nu$.  A similar situation holds for $\Irr(A_\nu|\nu)$.  So $|\Irr(G_\nu|\nu)|=|\Irr(G_\nu/\krn(\nu))|\nu')|$, where $\nu'$ is the character of $N/\krn(\nu)$ corresponding to $\nu$, and $N/\krn(\nu)$ is necessarily cyclic since $\nu$ is an irreducible linear character.  We now observe, via Eq. \eqref{cyceq} applied to $H$ and $A_\nu/\krn(\nu)$:
\[ \ccl(H) - \ccl(H/K) = (p-1)|\Irr(G_\nu|\nu)| \]
\[ \ccl(U) - \ccl(U/K) = (p-1)|\Irr(A_\nu|\nu)| \]
where $U$ is a Hall $\pi$-subgroup of $H$, and $K=N/\krn(\nu) \leq Z(H)$.  From the fact that $H$ is $\SConpi$, we obtain
\[ |\Irr(G_\nu|\nu)| \leq |\Irr(A_\nu|\nu)|\ccl(V) \]
where $V$ is a Hall $\pi'$-subgroup of $H$.  Since $V \cong S_\nu$, this is equivalent to inequality \eqref{baineq}.\\

Now suppose $\nu=1$.  The inequality \eqref{baineq} becomes:

\[ |\Irr(G|1_N)| \leq |\Irr(A|1_N)|\ccl(B) \]
which is equivalent to
\[ \ccl(G/N) \leq \ccl(A/N)\ccl(B) \]
By assumption, $G/N$ is $\Conpi$ so this inequality is satisfied.  This concludes the proof that $G$ is $\Conpi$.\\

Suppose $\ccl(G) = \ccl(A)\ccl(B)$, in either case (i) or case (ii).  Then we must have equality in \eqref{baineq} for each $\nu \in \Irr(N)$, and furthermore, by our use of Corollary \ref{sgpccl} in obtaining \eqref{baineq}, $S_\nu$ must be isomorphic to $B$, rather than to a proper subgroup of $B$; in other words $S_\nu$ is a Hall $\pi'$-subgroup of $G$.  In the case where $\nu$ is a $G$-invariant character (such as the trivial character), equality in \eqref{baineq} means $\ccl(G/N) = \ccl(A/N)\ccl(BN/N)$.  If $G/N$ is $\Conpi^*$, for this to happen $BN/N$ must centralise $A/N$ in $G/N$, which implies $A \lhd G$.  Hence $G$ is $\Conpi^*$, by Theorem A.\end{proof}

Case (ii) above depends on the condition $\SConpi$, which is stronger than $\Conpi$.  The following gives some circumstances under which $\SConpi$ can be obtained.

\begin{thmd}
Let $G$ be a finite soluble group.  Suppose that for some central subgroup $M$ of prime order, every subgroup of $G/M$ is $\Conpi$ and $M \cap G' = 1$.  Then $G$ is $\Conpi$.  If every central subgroup of $G$ of prime order satisfies these conditions, then $G$ is $\SConpi$.  If $G/M$ is $\Conpi^*$ then $G$ is $\Conpi^*$.
\end{thmd}

\begin{proof}
Without loss of generality we may assume $M \leq A$.  As $M$ is central, $G$ has balanced action on $\Irr(M)$.  Since $M \cap G' = 1$ and $G'$ is the intersection of the kernels of all linear characters of $G$, there is a linear character $\chi$ of $G$ which extends some $\mu \in \Irr(M) \setminus \{1_M\}$.  The same applies to all $\mu \in \Irr(M)$, since the non-trivial characters are all equivalent via a Galois automorphism, and $1_M$ extends to $1_G$.  We are now in the situation of case (i) of Theorem C, and so $G$ is $\Conpi$, and if $G/M$ is $\Conpi^*$ then $G$ is $\Conpi^*$.\\

Now assume every central subgroup of prime order satisfies the stated conditions.  To show $G$ is $\SConpi$, it suffices to consider a given $M \leq G$, assume $M \leq A$, and show
\[ \ccl(G) - \ccl(G/M) \leq (\ccl(A) - \ccl(A/M))\ccl(B) \]
which is equivalent by Eq. \eqref{cyceq} to showing the following, for any $\mu \in \Irr(M) \setminus \{1_M\}$:
\[ \Irr(G|\mu) \leq \Irr(A|\mu)\ccl(B) \]
But since $\mu$ extends to $G$ (and hence also to $A$), the above inequality is equivalent to the following:
\[ \ccl(G/M) \leq \ccl(A/M)\ccl(B) \]
This inequality follows immediately from the fact that $G/M$ is $\Conpi$.\end{proof}

\subsection{Theorem E}

Recall the bijection $*$ defined in the discussion for Theorem B, where it was noted that if $N \unlhd G$, then $N^* \unlhd G^*$.  However, in general $*$ only defines an injection from normal subgroups of $G$ to normal subgroups of $G^*$.  We associate to each $K \unlhd G^*$ a normal subgroup of $G$ as follows: define the \emph{$G$-level} $\Lev (K)$ to be the join of all $N$ such that $N \unlhd G$ and $N^* \leq K$.  $\Lev (K)$ is thus the largest normal subgroup of $L$ of $G$ such that $L^* \leq K$.\\

We may associate to any complex character $\chi$ of $G$ a normal subgroup of $G$, namely $\krn(\chi)$.  Given a complex character $\phi$ of $G^*$, we may also associate a normal subgroup of $G$ to it, namely $\Lev(\krn(\phi))$.  This gives a way of comparing irreducible complex characters of $G$ to those of $G^*$.  Given a set $X$ of normal subgroups of $G$, let $\Irr(G|X)$ denote those $\chi \in \Irr(G)$ for which $\krn(\chi) \in X$, and $\Irr^G(G^*|X)$ denote those $\phi \in \Irr(G^*)$ for which $\Lev(\krn(\phi)) \in X$.  If $X$ is a singleton $\{N\}$, write $\Irr(G|N)$ and $\Irr^G(G^*|N)$.\\

Using $G$-levels, we can potentially reduce the question of whether a finite soluble group is $\SConpi$ (and hence $\Conpi$) to a question about those quotients which are irreducible complex linear groups:

\begin{thme}Let $\mathcal{K}$ be a family of finite soluble groups closed under quotients.  Suppose that $|\Irr(G|1)| \leq |\Irr^G(G^*|1)|$ for every group $G \in \mathcal{K}$, where $*$ indicates the bijection given by $(ab)^*=(a,b)$ from $G$ to a direct product of a Hall $\pi$-subgroup $A$ and a Hall $\pi'$-subgroup $B$ of $G$.\\

Then every group in $\mathcal{K}$ is $\SConpi$.\end{thme}

\begin{proof}
Let $G \in \mathcal{K}$.  For every $N \unlhd G$, we have
\[ |\Irr(G|N)| = |\Irr(G/N|1)| \leq |\Irr^{G/N}(G^*/N^*|1)| = |\Irr^G(G^*|N)| \]
For every set $X$ of normal subgroups of $G$, we have
\[ |\Irr(G|X)| = \sum_{N \in X}|\Irr(G|N)| \]
and similarly
\[ |\Irr^G(G^*|X)| = \sum_{N \in X}|\Irr^G(G^*|N)| \]
from which we conclude that $|\Irr(G|X)| \leq |\Irr^G(G^*|X)|$ for any set $X$.  If $X$ is the set of all normal subgroups of $G$, this proves that $G$ is $\Conpi$.\\

Now suppose $M$ is a central subgroup of $G$ of prime order.  We may assume $M \leq A$.  We must show
\[ \ccl(G) - \ccl(G/M) \leq (\ccl(A) - \ccl(A/M))\ccl(B) \]
which follows from
\[ |\Irr(G|X)| \leq |\Irr^G(G^*|X)| \]
where $X$ is the set of normal subgroups of $G$ not containing $M$.  So $G$ is $\SConpi$.
\end{proof}

It is not known in general which families of groups have the required properties.  Indeed, it is not known whether there are any finite soluble groups which violate the conditions of Theorem E.  However, the following proposition suggests it may be helpful to use character degrees when trying to obtain bounds on the number of faithful characters.\\

\begin{prop}\label{sumchisq}
 Let $X$ be a set of normal subgroups of $G$.  Then
 
\[ \sum_{\chi \in \Irr(G|X)} \chi(1)^2 = \sum_{\phi \in \Irr^G(G^*|X)} \phi(1)^2 \]
\end{prop}

\begin{proof}
Assume $G$ is a minimal counterexample.  To prove the result for $X$, it is sufficient to consider singleton subsets $\{ N \}$ of $X$ individually.  If $N \not=1$, then 
\[ \sum_{\chi \in \Irr(G|N)} \chi(1)^2 = \sum_{\chi \in \Irr(G/N|1)} \chi(1)^2\]
\[ \sum_{\phi \in \Irr^G(G^*|N)} \phi(1)^2 = \sum_{\phi \in \Irr^G(G^*/N^*|1)} \phi(1)^2\]
By minimality of $|G|$,
\[ \sum_{\chi \in \Irr(G/N|1)} \chi(1)^2 = \sum_{\phi \in \Irr^G(G^*/N^*|1)} \phi(1)^2 \]\\

So the equation claimed in the Proposition holds when $X$ is the set of all non-trivial normal subgroups of $G$.  But then
\[ \sum_{\chi \in \Irr(G|X)} \chi(1)^2 + \sum_{\chi \in \Irr(G|1)} \chi(1)^2 = |G| = \sum_{\phi \in \Irr^G(G^*|X)} \phi(1)^2 + \sum_{\phi \in \Irr^G(G^*|1)} \phi(1)^2\]

Hence the result also holds for $X= \{1\}$.  This proves the result in all cases.\end{proof}

There are known results concerning finite soluble irreducible complex linear groups of small degree, which given the above proposition may have some bearing on this problem.  Of particular relevance here is the following:

\begin{thm}[Winter, \cite{10}]
Let $G$ be a soluble irreducible linear group of degree $n$.  Suppose that a Sylow $p$-subgroup of $G$ is not normal.  Then $n$ is divisible by a prime power $q > 1$ such that $q \equiv \epsilon$ modulo $p$, where $|\epsilon|\leq 1$.
\end{thm}

We use this to give a numerical condition, which may be useful in ruling out counterexamples of small order to the hypothesis of Theorem E.

\begin{cor}
Let $G$ be a finite soluble group such that $|\Irr(G|1)| > |\Irr^G(G^*|1)|$, with $*$ defined as before.  Let $\alpha$ be the set of primes $p$ for which $G$ has a non-normal Sylow $p$-subgroup.  Let $\sum_{\chi \in \Irr(G|X)} \chi(1)^2 = k$.  Then $k$ is expressible in two ways as the sum of squares:
\[ k = a^2_1 + \dots + a^2_l = b^2_1 + \dots + b^2_m \]
such that $l > m$, every $a_i$ and $b_j$ divides the order of $G$, and for every $a_i$ and every $p \in \alpha$, either $p$ divides $a_i$ or $a^2_i$ is divisible by a prime power $r>1$ such that $r \equiv 1$ modulo $p$.\\

If in addition, there is no normal subgroup $L$ of $G^*$ such that $\Lev(L)=1$ and $G^*/L$ is cyclic, then we may require $b_j > 1$ for each $b_j$.
\end{cor}

\begin{proof}
The numbers $a_i$ are the degrees of the faithful irreducible characters of $G$, and the $b_j$ are the degrees of characters in $\Irr^G(G^*|1)$.  The conditions on them follow from Winter's theorem and Proposition \ref{sumchisq}.
\end{proof}

\section{Soluble groups of small order}

This section concerns soluble groups of order at most $2000$.  Non-nilpotent groups up to this order have been enumerated, and are available for processing in computerised mathematics systems.  It is therefore possible to test the conjecture directly on these groups.\\

However, for all but seven values of $n$ at most $2000$, we can show that every soluble group of order $n$ is $\Conpi^*$ for every $\pi \subseteq \bP$, by application of the other results in this paper.  It is therefore only necessary to search the database for groups of the remaining seven orders.\\

In this section, the order of $G$ will be denoted $n$.  Note that we do not need to specify $\pi$, as all possible $\pi$ will be considered (with the usual observation that $\Conpi^*$ is equivalent to $\mathrm{Con}^*_{\pi'}$).

\begin{lem}\label{p3lem}
Suppose $A$ is either abelian, or of order $p^3$, where $p$ is a prime.  Then $G$ is $\Conpi^*$.
\end{lem}

\begin{proof}
Using Proposition \ref{easysgp}, it suffices to show that $\ccl(H) < \ccl(A)$ for every proper subgroup $H$ of $A$.  If $A$ is abelian this is obvious.  If $A=p^3$ and $A$ is non-abelian, we note that $Z(A) >1$, and the centraliser of every element properly contains $Z(A)$; in other words, every centraliser has order at least $p^2$.  This means each conjugacy class has size at most $p$, and there are central conjugacy classes, so $\ccl(A) > p^2$.  But $\ccl(H) \leq |H| \leq p^2$.\end{proof}

\begin{cor}\label{fewprimes}Suppose $G$ is not $\Conpi^*$ and the prime divisors of $|G|=n$ are a subset of $\{p,q,r\}$.  Then we may assume $\pi$ is a singleton, say $\pi=\{p\}$, and in this case $n$ is divisible by $p^4$ and $G$ does not have a normal Sylow $p$-subgroup.  If $r$ does not divide $n$, then $n$ is divisible by $p^4 q^4$ and $G$ has no non-trivial normal Sylow subgroups.  In particular, if $r$ does not divide $n$ and $n \leq 2000$, then $n=1296$.\end{cor}

Denote by $F(G)$ the Fitting subgroup of $G$, that is the largest normal niplotent subgroup.  Note that in a finite soluble group, the Fitting subgroup always contains its own centraliser.  We will denote by $m$ the order of $F(G)$.

\begin{lem}\label{fitlem}Let $S$ be a Sylow subgroup of $G$ of order $p$.  Suppose $G$ is a minimal soluble counterexample to $\Conpi^*$.  Then $p$ does not divide $m$, and $F(G)$ has a coprime automorphism of order $p$.\end{lem}

\begin{proof}Suppose $p$ divides $m$.  Then $S \leq F(G)$ and $S$ is normal in $G$.  Since $S$ is cyclic, $G$ has balanced action on $\Irr(S)$, and since $S$ is a normal Sylow subgroup, every irreducible character of $S$ extends to $G$.  A contradiction follows by applying Theorem C, case (i).\end{proof}

\begin{lem}\label{nlplem}Let $H$ be a finite nilpotent group.  Suppose the prime $p$ divides $|\Aut(H)|$, but not $|H|$.  Then $|H|$ is divisible by a prime power $r$ such that $p$ divides $r-1$.\end{lem}

\begin{proof}Any coprime automorphism of $H$ must act faithfully on $H/\Phi(H)$, where $\Phi(H)$ is the Frattini subgroup.  $H/\Phi(H)$ is the direct product of characteristic elementary abelian subgroups, one for each prime dividing its order, so $\Aut(H/\Phi(H))$ is the direct product of general linear groups.  (See \cite{4}.)  The result follows by considering the order of these general linear groups.\end{proof}

From now on, we make the additional assumptions that $n \leq 2000$, and that $G$ is a minimal counterexample to $\Conpi^*$.

\begin{lem}\label{peplem}(i)The order $n$ of $G$ cannot be divisible by $r=p^{e(p)}$ for any $p$, where $e(2)=7$, $e(3)=5$, $e(5)=4$, $e(7)=3$, and $e(p)=2$ otherwise.

(ii)The order $m$ of $F(G)$ divides $2^6.3^4.5^3.7^2$.

(iii)Let $H$ be a nilpotent group of order dividing $2^6.3^4.5^3.7^2$.  Then $\Aut(H)$ is a $\{2,3,5,7,13,31\}$-group.\end{lem}

\begin{proof}(i)Assume for a contradiction that $r$ divides $n$.  We may assume $p \in \pi$.  In each case, $n/r <16$.  So $|B| < 16$, and by Lemma \ref{p3lem} we may assume $B$ is non-abelian and $|B| \not= 8$.  Hence $|B| \in \{6,10,12,14\}$.  Clearly $p>2$, and $|A| \leq 333$, which rules out $p=5$ and $p=7$.  If $p=3$, we must have $|B| \in \{10,14\}$, so $|A| \leq 200$.  But $p^{e(p)}=243>200$ in this case.  So we may assume $p>7$. This means $|A|$ is at least $2p^2 > 200$ by Lemma \ref{p3lem}, and $|B|=6$.  Since $|A|$ is coprime to $2$ and $3$ and divisible by $p^2$, it must be at least $5p^2$.  But then $n = |A||B| \geq 30p^2 > 2000$, a contradiction.\\
 
(ii)For each prime $p>7$, we note $p^2$ does not divide $n$, so $p$ does not divide $m$.  The conclusion now follows from part (i).\\

(iii)This is a consequence of Lemma \ref{nlplem}.\end{proof}

\begin{prop}
Let $\pi$ be a set of primes.  Let $G$ be a finite soluble group of order at most $2000$, and suppose $G$ is a minimal counterexample to the assertion `every finite soluble group is $\Conpi^*$'.  Then $|G|$ is one of the following:
\[ 336, 672, 1008, 1200, 1296, 1344, 1680 \]
\end{prop}

\begin{proof}Assume $2 \in \pi$.  By previous results, $G$ is a $\{2,3,5,7,13,31\}$-group and $F(G)$ is a $\{2,3,5,7\}$-group.  We claim in fact $G$ is a $\{2,3,5,7\}$-group.  By Corollary \ref{fewprimes}, we may assume $n$ has at least three prime divisors.\\

If $31$ divides $n$ then $F(G)$ must admit an automorphism of order $31$.  The only possibility is if $2^5$ divides $m$.  But $2000/(2^5.31) < 3$, so no more prime divisors are possible.\\

If $13$ divides $n$ then $F(G)$ must admit an automorphism of order $13$.  The only possibility is if $3^3$ divides $m$.  By Corollary \ref{fewprimes}, $n$ must also either involve at least four primes, or involve three primes and be divisible by a fourth power.  Both however are impossible as $2000/(3^3.13) < 6$.\\

Suppose $n$ is divisible by more than three primes.  Then the prime divisors are $2$, $3$, $5$ and $7$.  If $7^2$ divides $n$, the only possibility is $n=2.3.5.7^2=1470$, but then $F(G)$ cannot have an automorphism of order $5$, contradicting of Lemma \ref{fitlem}.  Hence $7$ divides $n$ but not $7^2$, so $F(G)$ must have a coprime automorphism of order $7$, which implies $2^3$ divides $m$.\\

If $2^4$ divides $n$, then $n$ must be $2^4.3.5.7=1680$, so we may assume this does not occur.  Since $2^3.3^4.5.7 > 2000$ and $2^3.3.5^2.7 > 2000$, neither $3^4$ nor $5^2$ divides $n$. Hence $m$ divides $2^3.3^3$ and $F(G)$ cannot have an automorphism of order $5$, leading to a contradiction of Lemma \ref{fitlem}.\\

We may now assume $n$ has exactly three prime divisors.  By Corollary \ref{fewprimes}, $n$ is divisible by a fourth power $p^4$ say, such that $G$ does not have a normal Sylow $p$-subgroup, and Hall $p'$-subgroups are non-abelian.  By Lemma \ref{peplem}, $p$ must be $2$ or $3$, and $n$ is in addition divisible by $q$, where $q$ is $5$ or $7$.  If $p=3$, then $n/(p^4 q)$ is at most $4$.  As this is less than $q$, there must be a coprime automorphism of $F(G)$ of order $q$.  The Hall $3'$-subgroup of $F(G)$ cannot have an automorphism of order $q$, so the Hall $3$-subgroup of $F(G)$ must have an automorphism of order $q$.  As a result $3^4$ must divide $F(G)$, so $3^5$ divides $G$, contradicting Lemma \ref{peplem}.\\

We have now established $2^4$ divides $n$.  The numbers listed in the statement of the theorem include all multiples of $336 = 2^4.3.7$ which are at most $2000$, so we may assume $n$ is not divisible by $336$.  Hence we may assume that the prime divisors of $n$ are $2$ and $5$ and exactly one of $3$ and $7$.  In this case, a Hall $2'$-subgroup $H$ of $G$ must have order dividing $35$, $45$ or $75$.  All groups of order dividing $25$, $35$ or $45$ are abelian, so $|H|$ is $75=3.5^2$.  This means $n=2^4.3.5^2=1200$.\end{proof}

To complete the proof of Theorem F, we check for the seven group orders listed that every group is $\Conpi^*$, by calculating $\ccl(G) - \ccl(A)\ccl(B)$ directly for all relevant $\pi$.  This was done using \cite{8}.

\section{Acknowledgments}
This paper is based on results obtained by the author while under the supervision of Robert Wilson at Queen Mary, University of London.  The author would like to thank John Bray and Robert Wilson for their advice in obtaining these results, and Charles Leedham-Green for his detailed corrections and feedback regarding the presentation of this paper.

\end{document}